\documentclass[a4paper, 12pt]{article}
\usepackage{amsmath}
\usepackage{amssymb,amsthm, amsfonts}
\usepackage{mathrsfs}
\usepackage{bm}
\usepackage{graphicx}
\usepackage{color}
\usepackage{multicol}
\usepackage{enumerate}
\usepackage{easyReview}
\usepackage{subfig}
\usepackage{soul}
\usepackage{authblk}
\usepackage{cite}
\usepackage{hyperref}

\newtheorem{theorem}{Theorem}[section]
\newtheorem{corollary}[theorem]{Corollary}
\newtheorem{lemma}[theorem]{Lemma}
\newtheorem{remark}[theorem]{Remark}

\newcommand{\R}{\mathbb{R}}

\newcommand{\N}{\mathcal{N}}

\newcommand{\dx}{\mathrm{d}x}
\newcommand{\la}{\langle}
\newcommand{\ra}{\rangle}

\usepackage{hyperref}
\hypersetup{
	colorlinks=true,
	linkcolor=blue,
	filecolor=blue,
	urlcolor=blue,
	citecolor=cyan,
}
\everymath{\displaystyle}

\numberwithin{equation}{section}

\title{Ground state of indefinite coupled nonlinear Schr\"odinger systems}
\author{Ruijin Xu,  Jiabao Su$^\ast$, Rushun Tian\thanks{Corresponding authors: sujb@cnu.edu.cn, rushun.tian@cnu.edu.cn} \\{\small  School of Mathematical Sciences,  Capital Normal	University}\\{  \small Beijing 100048, People's Republic of China}}
\date{}
\begin{document}
	\maketitle
	\begin{abstract}
		{In this paper, we study the ground state solutions of the following} coupled nonlinear Schr\"odinger system
		\begin{equation*}
			\begin{cases}
				-\Delta u_1-\tau_1 u_1 =\mu_1u_1^3+\beta u_1u_2^2 & \text { in } \Omega,\\
				-\Delta u_2-\tau_2 u_2 =\mu_2u_2^3+\beta u_1^2u_2 & \text { in } \Omega,\\
				u_1=u_2=0 & \text { on } \partial\Omega.
			\end{cases} \eqno{(P)}
		\end{equation*}
		where $\mu_1, \mu_2>0$, $\beta>0$ and $\Omega\subset \R^N (N\le3)$ is a bounded domain with smooth boundary. We are concerned with the indefinite case, i.e., $\tau_1, \tau_2$ are greater than or equal to the principal eigenvalue of $-\Delta$ with the Dirichlet boundary datum. By delicate variational arguments, we obtain the existence of ground state solution to $(P)$, and also provide information on critical energy levels for coupling parameter $\beta$ in some ranges.
	\end{abstract}
	
	

	\section{Introduction}
	We are concerned with the following Schr\"odinger type elliptic system
	\begin{equation}\label{sy1}
		\begin{cases}
			-\Delta u_1-\tau_1 u_1=\mu_1u_1^3+\beta u_1u_2^2 & \text { in } \Omega,\\
			-\Delta u_2-\tau_2 u_2=\mu_2u_2^3+\beta u_1^{2}u_2 & \text { in } \Omega,\\
			u_1=u_2=0 & \text { on } \partial\Omega,
		\end{cases}
	\end{equation}
	where $\tau_i\in\R, \mu_i>0, i=1,2$, $\beta>0$ and $\Omega\subset \R^N(N\le3)$ is a bounded domain with smooth boundary $\partial \Omega$.

	In the last two decades, the elliptic system \eqref{sy1}, with various sets of parameters, has received much attention. The established results can be roughly characterized into two categories: the definite case and the indefinite case. If  either (a) $\Omega=\R^N$ and $\tau_i<0$, or (b) $\Omega\subset\R^N$ is bounded and $\tau_i<\lambda_1$, where $\lambda_1>0$ is the principal eigenvalue of $-\Delta$ in $H_0^1(\Omega)$, then the system \eqref{sy1} is definite. In contrast, as our main concerns, the indefinite system \eqref{sy1} is defined on a smooth bounded domain with $\tau_i\geq\lambda_1$. In the definite case, the associated Nehari manifold is complete and the energy functional has a lower bound on this manifold. Thus one may study the infimum of the energy functional constrained on the Nehari manifold, and the corresponding solution is known as the ground state solution. In the indefinite case, on the other hand, the ground state solution of \eqref{sy1} should be understood as the solution that achieves the minimal energy among all solutions.
	
	
	In the pioneering work \cite{lin_ground_2005}, Lin and Wei advanced the study of ground state solutions to \eqref{sy1} in the definite case. Since then, the existence of ground state solutions for various definite coupled Schr"odinger systems has been extensively studied, see \cite{ambrosetti_stading_2007,bartsch_note_2006,correia_semitrivial_2016,liu_ground_2010,soave_new_2016}.
	

	On the other hand, studies devoted to the indefinite case appear to be much fewer. The incompleteness of the Nehari manifold and a more complicated space decomposition introduce significant complexities. In \cite{szulkin_ground_nodate-1}, Szulkin and Weth proved that the indefinite scalar problem
	\begin{equation}\label{eq:scalar}
		-\Delta u - \tau_i u=\mu_i u^3,\quad u_i\in H_0^1(\Omega),\ i=1, 2,
	\end{equation}
	has a ground state solution $u_i\neq 0$, where $\tau_i\in \mathbb{R}$ and $\mu_i>0$. Consequently, the system \eqref{sy1} admits two types of semi-trivial solutions, $(u_1,0)$ and $(0,u_2)$.	Recently, in \cite{clapp_solutions_2020}, Clapp and Szulkin proved that a ground state solution exists when $\beta>0$ is sufficiently large, using a critical point theorem established by Bartolo, Benci and Fortunato \cite[Theorem 2.4]{Bartolo-etal_1983}. Husaini and Liu \cite{ali_husaini_synchronized_2022} extended the result of \cite{clapp_solutions_2020} to systems with general nonlinearities.
	
	
	In this paper, we investigate the existence of a ground state solution with both components being nontrivial. It is worth emphasizing that our method provides a specific range for $\beta$, which improves upon the result in \cite{clapp_solutions_2020}. In contrast to our recent work \cite{xst2025}, which focused on the case of small $|\beta|$, we develop a novel variational framework here to address the challenges arising from the altered geometric structure of the problem when $\beta$ is large.
	
	
	\subsection{Notations and main results}
	The norms of $L^p(\Omega)$ and $H_0^1(\Omega)$ are denoted by
	\begin{equation*}
		\|u\|_{L^p}=\left( \int_{\Omega} |u|^p  \dx\right)^\frac{1}{p},\quad \|u\|_{H^1_0}=\left( \int_{\Omega}|\nabla u|^2  \dx \right)^\frac{1}{2},
	\end{equation*} respectively.
	For  $\mathbf{u}\in \mathbf{H}:=H^1_0(\Omega)\times H^1_0(\Omega)$, we write $\mathbf{u}=(u_1,u_2)$ and define the norm
	$$\|\mathbf{u}\|=(\| u_1\|_{H_0^1}^2+ \| u_2\|_{H_0^1}^2 ) ^\frac{1}{2}.$$
	Based on the spectrum of $-\Delta-\tau_i$ in $H_0^1(\Omega)$, we adopt the following orthogonal decomposition for $i=1,2$:
	$$H^1_0(\Omega)=H_{i}^+ \oplus H_{i}^0 \oplus H_{i}^-,$$
	where the operator $-\Delta-\tau_i$ is positive definite on $H_i^+$, zero on $H_i^0$ and negative definite on $H_i^-$.
	Let $\tilde{H}_i=H^0_i\oplus H^-_i$, then we also have
	$$H^1_0(\Omega)=H_{i}^+ \oplus\tilde{H}_i,\ i=1,2.$$
	We set
	$$\mathbf{H}^+:=H_1^+\times H_2^+,\ \mathbf{H}^0:=H^0_1\times H^0_2,\ \mathbf{H}^-:=H_1^-\times H_2^-,\ \tilde{\mathbf{H}}:=\tilde{H}_1\times\tilde{H}_2.$$
	Then  any $\mathbf{u}\in\mathbf{H}$  admits a unique orthogonal decomposition
	\begin{align*}
		\mathbf{u}& =\mathbf{u}^++\mathbf{u}^0+\mathbf{u}^-\ \in \mathbf{H}^+\oplus \mathbf{H}^0\oplus \mathbf{H}^-\\
		&=\mathbf{u}^+ +\tilde{\mathbf{u}}, \ \  \ \ \   \tilde{\mathbf{u}} = \mathbf{u}^0+\mathbf{u}^-  \in  \tilde{\mathbf{H}}.
	\end{align*}
	We denote
	$$\mathbf{u}=(u_1,u_2),\ \mathbf{u}^+=(u_1^+,u_2^+),\ \mathbf{u}^-=(u_1^-,u_2^-),\ \tilde{\mathbf{u}}=(\tilde{u}_1,\tilde{u}_2).$$
	The weak solutions of \eqref{sy1} correspond to critical points of the energy functional
	$I: \mathbf{H} \rightarrow\R$,
	\begin{align*}
		I(\mathbf{u}):&=\frac{1}{2}\int_\Omega(|\nabla u_1|^2+|\nabla u_2|^2-\tau_1u_1^2-\tau_2u_2^2) \dx-
		\frac{1}{4}\int_\Omega (\mu_1u_1^4+\mu_2u_2^4+2\beta u_1^2u_2^2) \dx.
	\end{align*}
	For simplicity, we set 	$$
	J(\mathbf{u},\mathbf{v}):=J_1(u_1,v_1)+J_2(u_2,v_2)$$
	where $$ 	J_i(u_i,v_i):=\int_\Omega\nabla u_i\cdot\nabla v_i \dx-\tau_i\int_\Omega u_iv_i \dx, \ i=1, 2. $$ 	Let
	\begin{align*} 		F(\mathbf{u}):=\frac{1}{4}(\mu_1u_1^4+\mu_2u_2^4+2\beta u_1^2u_2^2),\end{align*} and
	$$f(\mathbf{u}):=\nabla F(\mathbf{u})=(\mu_1u_1^3+\beta u_1u_2^2,\mu_2u_2^3+\beta u_1^2u_2).$$
	Then we can rewrite the functional $I$ as
	$$
	I(\mathbf{u})=\frac{1}{2}J(\mathbf{u},\mathbf{u})-\int_{\Omega}F(\mathbf{u}) \dx.$$
	If $\mathbf{u}^*\in \mathbf{H}$ achieves the minimal energy   $e$ given by
	\begin{equation}\label{eq:groundEnergy}
		e:=\inf\limits_{\mathbf{u}\in \mathcal{K}}I(\mathbf{u})=I(\mathbf{u}^*),
	\end{equation} among all nontrivial critical points of $I$, where
	$\mathcal{K}=\{\mathbf{u}\in\mathbf{H} \backslash \tilde{\mathbf{H}} : I'(\mathbf{u})=0\}$, then $\mathbf{u}^*$ is called a ground state solution of \eqref{sy1}.

	To seek ground state solution of \eqref{gs}, the following constraints and infima will be used.
	\begin{enumerate}
		\item[(1)] The so-called \emph{Nehari-Pankov manifold} associated to $I$,
		\begin{equation*}
			\N:=\{\mathbf{u}\in \mathbf{H}\backslash\tilde{\mathbf{H}}:
			I'(\mathbf{u})\mathbf{u}=0,I'(\mathbf{u})\mathbf{v}=0\text{ for all } \mathbf{v}\in\tilde{\mathbf{H}}\},
		\end{equation*}
		which was first introduced by Pankov in \cite{pankov_periodic_2005}. Denote
		\begin{equation}\label{eq:level_c}
			c:=\inf\limits_{\mathbf{u}\in\N}I(\mathbf{u}).
		\end{equation}
		\item[(2)] The set of \emph{maximal points on generalized fibers},
		\begin{equation*}
			\N':=\left\{\mathbf{u}\in \mathbf{H}\backslash\tilde{\mathbf{H}} :I(\mathbf{u})\ge
			I(t\mathbf{u}+\mathbf{v})\text{ for all } t\ge0, \mathbf{v}\in\tilde{\mathbf{H}}\right\},
		\end{equation*}
		which coincides with $\N$ if $\tilde{\mathbf{H}}=\{\mathbf{0}\}$. Denote
		\begin{equation}\label{eq:level_cprime}
			c':=\inf\limits_{\mathbf{u}\in\N'}I(\mathbf{u}).
		\end{equation}
	\end{enumerate}
	{Obviously, since $I$ is an even functional, $\N'$ consists of the global maximizer of $I|_{\R\mathbf{u}\oplus \tilde{\mathbf{H}}}$, $\N'\subset \N$.} Moreover, $\N$ contains all the nontrivial critical points for $I$, so $\mathcal{K}\cup\N'\subset\N$ and
	\begin{equation}\label{cineq}
		c\le \min\{e,c'\} .
	\end{equation}
	Let $K_i$ be the set of ground state solutions of \eqref{eq:scalar}. By \cite{szulkin_ground_nodate-1}, $K_i\neq \emptyset$. Define
	\begin{equation} \label{1.7}
		\hat{\beta}_1=\inf_{U_1\in K_1}\inf _{\varphi \in H_2^+ \backslash\{0\}} \frac{J_2(\varphi,\varphi)}{\int_\Omega U_1^2 \varphi^2 \dx},
		\quad \hat{\beta}_2=\inf_{U_2\in K_2}\inf _{\varphi \in H_1^+ \backslash\{0\}} \frac{J_1(\varphi,\varphi)}{\int_\Omega U_2^2 \varphi^2 \dx}.
	\end{equation}
	
	Similar numbers have been defined by Ambrosetti and Colorado for the definite case  in \cite{ambrosetti_stading_2007}. Note that both $\hat{\beta_1}$ and $\hat{\beta_2}$ are well-defined and have positive lower bound. In fact, there exists constant $C'>0$ such that
	$$J_2(\varphi,\varphi)\ge C'\|\varphi\|_{H_0^1}^2,\quad \varphi\in H_2^+ \backslash\{0\}.$$
	Then for any $U_1\in K_1$, we have
	$$\frac{J_2(\varphi,\varphi)}{\int_\Omega U_1^2 \varphi^2 \dx}\ge\frac{C'\|\varphi\|_{H_0^1}^2}{\int_\Omega U_1^2 \varphi^2 \dx}\ge \frac{C'\|\varphi\|_{H_0^1}^2}{C^2\|U_1\|^2_{L^4}\|\varphi\|^2_{H_0^1}}=\frac{C'}{C^2\|U_1\|^2_{L^4}}>0.$$
	Since $\|U_1\|_{L^4}^2$ is independent of the choice of $U_1\in K_1$, {we have that  $\hat{\beta_1} >0$. Similarly, $\hat{\beta_2} >0$.}
	
	Let $c_{sem}$ be the least energy of the semi-trivial solutions of \eqref{sy1}, i.e., for $U_1\in K_1$, $U_2\in K_2$,
	\begin{equation}\label{eq:level_csem}
		c_{sem}:=\min\left\{I(U_1,0), I(0,U_2)\right\}.
	\end{equation}
	
	\begin{theorem}\label{gs}
		Assume that $\beta> \Lambda:=\max \left\{\hat{\beta}_1, \hat{\beta}_2\right\}$. Then the system \eqref{sy1} has a ground state solution $\mathbf{u^*}$. 	
		Moreover, it holds that
		$$ I(\mathbf{u^*})=e\le c_l\le c'<c_{sem},$$
		in which $e$, $c'$ and $ c_{sem}$ are defined in  \eqref{eq:groundEnergy}, \eqref{eq:level_cprime} and \eqref{eq:level_csem}, respectively, and
		\begin{equation}\label{eq:linkingvalue}
			c_l:=\inf\limits_{\mathbf{u}\in \mathbf{H}\backslash \tilde{\mathbf{H}}}\inf\limits_{\gamma\in \Gamma(\mathbf{u})}
			\max\limits_{\mathbf{u'}\in M(\mathbf{u})}I(\gamma(\mathbf{u'})),
		\end{equation}
		where  $$M(\mathbf{u})=\left\{\mathbf{w}=t\mathbf{u}+\mathbf{v}:\|\mathbf{w}\|\le \rho(\mathbf{u}),\ 	t\ge0 \text{ and }\mathbf{v} \in \tilde{\mathbf{H}}\right\},$$ $$\Gamma(\mathbf{u})=\left\{\gamma\in C(M(\mathbf{u}),\mathbf{H}): \gamma|_{\partial M(\mathbf{u})}=id\right\},$$
		 {and $\rho(\mathbf{u})>0$ will be defined in Lemma \ref{l-g}.}
	\end{theorem}
	
	\begin{remark}
		For the definite system on $\mathbb{R}^N$, the number $c_l$ is known to be a variant of the mountain pass level. Indeed, \cite[Lemma 3.2]{maia_positive_2006} establishes that $c=c_l=c'$, a value which coincides with the ground state energy level for $\beta>0$. The lower bound $\Lambda$ we obtain in Theorem \ref{gs} for the indefinite case is markedly sharper than the one implied by the existence result in \cite{clapp_solutions_2020}. This improvement stems from our refined and fundamentally different variational approach.
	\end{remark}

	Next, we present some results for a special indefinite case.
	\begin{theorem}\label{inf} 		Assume that $\tau_1=\tau_2=\lambda_1$ and $0<\beta<3\sqrt{\mu_1\mu_2}$.
		Then there holds that $\N=\N'$ and  the numbers \begin{equation} \label{crivaluesid1}
			c=e=c_l=c'
		\end{equation}  is achieved. Moreover, there exists a minimizer of \eqref{crivaluesid1}, which is a synchronized solution to the system (\ref{sy1}), i.e., a solution of the form $(t_1u, t_2u)$ with $t_1,t_2\in \R$.
	\end{theorem}
	\begin{remark}\label{rem:1-4}  
		We remark that the equality \(\mathcal{N} = \mathcal{N}'\) also holds under the following conditions:
		\begin{enumerate}
			\item[(1)] \(\tau_1, \tau_2 \in \mathbb{R}\) and \(0 < \beta < \sqrt{\mu_1 \mu_2}\);
			\item[(2)] either \(\tau_1 < \lambda_1\) or \(\tau_2 < \lambda_1\), and \(\beta > 0\);
			\item[(3)] \(\tau_1 < \lambda_1\), \(\tau_2 < \lambda_1\), and \(\beta \in \mathbb{R}\).
		\end{enumerate}
		It should be pointed out that without any restrictions on \(\beta\), the equality \(\mathcal{N} = \mathcal{N}'\) may fail. To illustrate this, we present an example where \(\mathcal{N}' \subsetneq \mathcal{N}\).
		
		Consider the case \(\tau_1 = \tau_2 = \lambda_1\) and \(\beta > \max\{\mu_1, \mu_2\}\). In this setting, system \eqref{sy1} admits a synchronized solution
		\[
		(\alpha_1 \omega, \alpha_2 \omega) := \left( \sqrt{\frac{\mu_2 - \beta}{\mu_1 \mu_2 - \beta^2}} \, \omega, \sqrt{\frac{\mu_1 - \beta}{\mu_1 \mu_2 - \beta^2}} \, \omega \right),
		\]
		where \(\omega\) is the ground state of the scalar problem
		\[
		-\Delta u - \lambda_1 u = u^3, \quad u \in H^1_0(\Omega).
		\]
		It is well-known that \((\alpha_1 \omega, \alpha_2 \omega) \in \mathcal{N}\).
		Let \(\varphi_1\) be the eigenfunction of \(-\Delta\) corresponding to the principal eigenvalue \(\lambda_1\). By the unique continuation principle, \(\varphi_1 \neq 0\) almost everywhere, and hence
		\[
		\int_{\Omega} \omega^2 \varphi_1^2 \, dx > 0.
		\]
		Observe that as \(\beta \to +\infty\), there hold $\alpha_1, \alpha_2 \to 0$ and $\sqrt{\beta} \, \alpha_1, \sqrt{\beta} \, \alpha_2 \to 1$.
		Consequently,
		\begin{align*}
			&\bigl\langle I''(\alpha_1 \omega, \alpha_2 \omega)(\varphi_1, -\varphi_1), (\varphi_1, -\varphi_1) \bigr\rangle \\
			&\quad = -3 \int_{\Omega} (\mu_1 \alpha_1^2 + \mu_2 \alpha_2^2) \omega^2 \varphi_1^2 \, dx - \beta \int_{\Omega} (\alpha_1^2 + \alpha_2^2 - 4 \alpha_1 \alpha_2) \omega^2 \varphi_1^2 \, dx \\
			&\quad \to 2 \int_{\Omega} \omega^2 \varphi_1^2 \, dx > 0, \quad \text{as } \beta \to +\infty.
		\end{align*}
		Therefore, for \(\beta > 0\) sufficiently large,
		\[
		\bigl\langle I''(\alpha_1 \omega, \alpha_2 \omega)(\varphi_1, -\varphi_1), (\varphi_1, -\varphi_1) \bigr\rangle > 0,
		\]
		which implies \((\alpha_1 \omega, \alpha_2 \omega) \notin \mathcal{N}'\).
	\end{remark}
	
	While, when $\beta>0$ is  large enough, the conclusion is quite different, the equality \eqref{crivaluesid1} no longer hold.  More precisely, we have the following theorm.
	\begin{theorem}\label{cpr}
		Assume that $\tau_1=\tau_2=\lambda_1$ and $\beta>0$ is large enough. Then the numbers defined by (1.3), (1.4) and (1.5) satisfy $c\le e<c'$.
	\end{theorem}
	
	This paper is organized as follows. In Section 2, we establish a few of lemmas to find a linking solution on the level $c_l$ and prove Theorem \ref{gs}. In Section 3, we give the proofs for Theorems \ref{inf} and \ref{cpr}, which reveal the influence of the coupling parameter $\beta$. We always assume $\beta>0$ hereafter.	
	
	\section{Existence of fully nontrivial ground state solution}

	We shall prove the existence of a solution to \eqref{sy1} by linking-type arguments.  For a given $\mathbf{u}\in\mathbf{H}$, we construct the subspace $\mathbf{H}(\mathbf{u})=\R\mathbf{u}\oplus \tilde{\mathbf{H}}$ and the convex subset  $\hat{\mathbf{H}}(\mathbf{u})=\R^+ \mathbf{u}\oplus \tilde{\mathbf{H}}$ where $\R^+=[0,+\infty)$. The following lemma is essential to verify the linking geometry. We point out here that the scalar case of this lemma has been proved in \cite[Lemma 2.5]{szulkin_ground_nodate-1}.
	\begin{lemma} \label{compact} Let $E\subset \mathbf{H}^+\backslash\{\mathbf{0}\}$ be  a compact set. Then  there exists a $R>0$ such that
		\begin{equation*}
			\sup_{\mathbf{u}\in E}\sup_{\mathbf{w}\in \mathbf{H}(\mathbf{u})\backslash B_R(\mathbf{0})} I(\mathbf{w})\leq 0.
		\end{equation*}
	\end{lemma}
	\begin{proof}
		If not, then there are sequences $\{\mathbf{u}_n\}\subset E$ and $\{\mathbf{w}_n\}\subset \mathbf{H}(\mathbf{u}_n)$ such that
		$$ I(\mathbf{w}_n)>0, \ \ \  \| \mathbf{w}_n\|\rightarrow \infty, \ n\to\infty.$$  Let  $$\mathbf{v}_n=\frac{\mathbf{w}_n}{\|\mathbf{w}_n\|}=(v_{n,1}, v_{n,2}) = (t_nu_{n,1}+\tilde{v}_{n,1},t_nu_{n,2}+\tilde{v}_{n,2}), \ \ \tilde{\mathbf{v}}_n=(\tilde{v}_{n,1}, \tilde{v}_{n, 2}).$$
		Since $E$ is a compact set, there exists $C_1>0$ such that
		$\|\mathbf{u}_n\|\ge C_1>0$. There holds
		$$\max\{t_n^2C_1^2,\|\tilde{\mathbf{v}}_n\|^2\}\le\max\{t_n^2\|\mathbf{u}_n\|^2,\|\tilde{\mathbf{v}}_n\|^2\}\le \|\mathbf{v}_n\|^2= 1. $$ Hence the sequences $\{t_n\}$, $\{\tilde{v}_{n,1}\}$ and $\{\tilde{v}_{n,2}\}$ are bounded, up to subsequences, we get $t_{n}\rightarrow t_{0}$ and $u_{n,i}\rightarrow u_i$, $\tilde{v}_{n,i}\rightarrow \tilde{v}_{0,i}$ in $H_0^1(\Omega)$ as $n\rightarrow \infty$ for $i=1,2$.  Set  $\mathbf{v_0}=(v_{0,1},v_{0,2})$, where $v_{0,i}=t_{0}u_{i}+\tilde{v}_{0,i}$.
		Then $\mathbf{v}_n=(v_{n,1}, v_{n,2}) \to \mathbf{v}_0$ in $\mathbf{H}$.  Consequently, $\| \mathbf{v}_0\|=1$. Without loss of generality, we assume $\| v_{0,1}\|_{H^1_0}\ge\| v_{0,2}\|_{H^1_0}$. Then $v_{0,1} \ne 0$ and $|w_{n,1}(x)|\rightarrow \infty$ if $v_{0,1}(x)\neq 0$.  By Fatou's lemma, we have
		\begin{align*}
			\frac{1}{4}\int_\Omega\left(\frac{\mu_1w_{n,1}^4+\mu_2w_{n,2}^4+2\beta w_{n,1}^2w_{n,2}^2}{\| \mathbf{w}_n\|^2}\right) \dx\ge \frac{\mu_1}{4}\int_\Omega\frac{w_{n,1}^4}{\| \mathbf{w}_n\|^2} \dx=\frac{\mu_1}{4}\int_\Omega w_{n,1}^2v_{n,1}^2 \dx\rightarrow \infty.
		\end{align*}
		Dividing $I(\mathbf{w}_n)$ by $\| \mathbf{w}_n\|^2$ and sending $n\to \infty$, we derive
		\begin{align*}
			0< \frac{I(\mathbf{w}_n)}{\| \mathbf{w}_n\|^2}&=\frac{1}{2}J(\mathbf{v}_n,\mathbf{v}_n)-\frac{1}{4}\int_\Omega\left(\frac{\mu_1w_{n,1}^4+\mu_2w_{n,2}^4+2\beta w_{n,1}^2w_{n,2}^2}{\| \mathbf{w}_n\|^2}\right) \dx
			\\&\le \frac{1}{2}\max\left\{1+\frac{|\tau_1|}{\lambda_1},1+\frac{|\tau_2|}{\lambda_1}\right\}-\frac{1}{4}\int_\Omega\left(\frac{\mu_1w_{n,1}^4+\mu_2w_{n,2}^4+2\beta w_{n,1}^2w_{n,2}^2}{\| \mathbf{w}_n\|^2}\right) \dx
			\\&\rightarrow -\infty,
		\end{align*}
		which is a contradiction.
	\end{proof}
	
	\begin{corollary}\label{maximum}
		For every $\mathbf{u}\in \mathbf{H}\backslash\tilde{\mathbf{H}}$, the set $\N'\cap \hat{\mathbf{H}}(\mathbf{u})$
		contains one point $\hat{\mathbf{u}}$ which is a global maximum point of $I|_{\hat{\mathbf{H}}(\mathbf{u})}$ and $I(\hat{\mathbf{u}})>0$.
	\end{corollary}
	\begin{proof}
		Fix $\mathbf{u}\in \mathbf{H}\backslash\tilde{\mathbf{H}}$, then $\mathbf{u}^+\ne 0$. We take  $E=\{\mathbf{u}^+\} \subset \mathbf{H}^+\backslash\{\mathbf{0}\}$ in Lemma \ref{compact}, then there exists a $R>0$ such that
		\begin{equation*}
			\sup_{\mathbf{v}\in \mathbf{H}(\mathbf{u})\backslash B_R(\mathbf{0})} I(\mathbf{v})\leq 0.
		\end{equation*}
		It is easy to see that  $I(s\mathbf{u}^+)>0$ when $s>0$ is small enough. Thus
		$$0<\sup\limits_{\hat{\mathbf{H}}(\mathbf{u})} I<\infty.$$
		Since $\hat{\mathbf{H}}(\mathbf{u})\cap B_R(\mathbf{0})$ is a bounded closed subset in a finite dimensional space, $\sup\limits_{\hat{\mathbf{H}}(\mathbf{u})}I$ is achieved at some point $\hat{\mathbf{u}} \in \N'\cap \hat{\mathbf{H}}(\mathbf{u})$.
	\end{proof}
	
	The next lemma concerns with the other requirement of a linking type geometry.
	{\begin{lemma}\label{l-g} For $r>0$ is small and for any $ \mathbf{u}\in \mathbf{H}\backslash \tilde{\mathbf{H}}$,
			there exists $\rho(\mathbf{u})>r$ such that
			 {$$\inf_{  S_r^+} I>0=\inf\limits_{\mathbf{u}\in \mathbf{H}\backslash\tilde{\mathbf{H}}} \max\limits_{\mathbf{w}\in\partial M(\mathbf{u})}I(\mathbf{w}),$$}
			where $S_r^+=\{\mathbf{z}\in \mathbf{H}^+:\|\mathbf{z}\|=r\}$ and
			$$	M(\mathbf{u})=\{\mathbf{w}=t\mathbf{u}+\mathbf{v}:\|\mathbf{w}\|\le \rho(\mathbf{u}),\ 	t\ge0\text{ and }\mathbf{v}\in \tilde{\mathbf{H}}\}. $$
		\end{lemma}
		\begin{proof} {Since $\int_{\Omega}F(\mathbf{z})dx=o(\|\mathbf{z}\|^2)$ as $\|\mathbf{z}\| \rightarrow 0$}, there are constants $\alpha>0$ and $r>0$ such that
			\begin{equation*}
				\inf_{\mathbf{z}\in S_r^+}I(\mathbf{z}) \geq \alpha >0, \quad \text{where} \  S_r^+=\{\mathbf{z} \in \mathbf{H}^+:\|\mathbf{z}\|=r\}.
			\end{equation*}
			By the proof of Lemma \ref{compact}, for any  given $\mathbf{u}\in \mathbf{H}\backslash \tilde{\mathbf{H}}$, there exists $\rho:=\rho(\mathbf{u})>r$ such that
			\begin{equation*}
				I (\mathbf{w})\le0\quad \text{for any } \mathbf{w}\in \partial M(\mathbf{u}) \text{ with } \|\mathbf{w}\|=\rho\text{ and }t\ge0.
			\end{equation*}
			For $\mathbf{w}\in \partial M(\mathbf{u})$ such that $\|\mathbf{w}\|\le \rho \text{ and }t=0$, then $\mathbf{w}\in \tilde{\mathbf{H}}$, we also have
			$$I (\mathbf{w})=\frac{1}{2}J(\mathbf{w},\mathbf{w})-\int_\Omega F(\mathbf{w})dx \le 0.$$
			For any $\mathbf{u}\in \mathbf{H}\backslash\tilde{\mathbf{H}}$, we get $\mathbf{0}\in \partial M(\mathbf{u})$. So
			$ {\inf\limits_{\mathbf{u}\in \mathbf{H}\backslash\tilde{\mathbf{H}}} \max\limits_{\mathbf{w}\in\partial M(\mathbf{u})}I(\mathbf{w})=I(\mathbf{0})=0}.$
	\end{proof}}
	
	The following linking-type lemma is inspired by \cite[Theorem 2.1] {mederski_ground_2016}.
	
	\begin{lemma}\label{ct}
		There exists a Palais-Smale sequence $\{\mathbf{u}_n\}$ for $I$ at the level $c_l$ defined by \eqref{eq:linkingvalue}. Moreover, there holds that $c_l\le c'$.
	\end{lemma}
	\begin{proof} We apply Lemma \ref{l-g} and the arguments in  the  proof of \cite[Theorem 2.12]{willem_minimax:1996}.
		For each fixed $\mathbf{u}\in \mathbf{H}\backslash \tilde{\mathbf{H}}$ and  for any $\gamma \in \Gamma(\mathbf{u})$, we have
		$$\max\limits_{M(\mathbf{u})}I\circ\gamma\ge \inf_{S_r^+}I.$$
		Therefore $ c_l\ge \inf_{S_r^+} I$, where $r>0$ is given in Lemma \ref{l-g} that is independent of $\mathbf{u}$.
		
		Now we assume for contradiction that there exists $\epsilon\in(0, c_l/2)$ such that
		for any $\mathbf{u}$ satisfying $c_l-2\epsilon\le I(\mathbf{u})\le c_l+2\epsilon$, there holds
		$$   \|I'(\mathbf{u})\|\ge \epsilon.$$
		By \cite[Theorem 2.3]{willem_minimax:1996}, we have a flow $\eta\in C([0,1]\times \mathbf{H},\mathbf{H})$ such that
		\begin{align*}
			\eta(t,\mathbf{u})=\mathbf{u}, &\quad \text{if } t=0\text{ or } \mathbf{u}\not\in I^{-1}([c_l-2\epsilon,c_l+2\epsilon])\\
			I(\eta(t,\mathbf{u}))\le I(\mathbf{u}),\ &\quad \forall \ t\ge0, \mathbf{u}\in \mathbf{H},\\
			I(\eta(1,\mathbf{u}))\le c_l-\epsilon,\ &\quad \forall \ \mathbf{u} \in I^{-1}((-\infty,c_l+\epsilon]).
		\end{align*}
		By the definition of $c_l$,   there exists $\mathbf{u}\in \mathbf{H}\backslash \tilde{\mathbf{H}}$ and $\gamma\in \Gamma(\mathbf{u})$, such that
		$$ I(\gamma(\mathbf{u'}))<c_l+\epsilon \ \ \hbox{for all} \ \mathbf{u}' \in M(\mathbf{u}).$$
		Note that $\eta(1,\gamma(\cdot)) \in \Gamma(\mathbf{u})$, it follows that
		$$ I(\eta(1,\gamma(\mathbf{u'}))) \le c_l-\epsilon,  \  \ \hbox{for all} \ \mathbf{u}' \in M(\mathbf{u}), $$
		which leads to a contradiction that  would read as $c_l \leq c_l-\epsilon$.
		
		At last, for any $\mathbf{u}\in \N'$, we take $\gamma=id\in \Gamma(\mathbf{u})$ and obtain
		$$c_l\le\max\limits_{\mathbf{u'}\in M(\mathbf{u})}I(\mathbf{u}')=I(\mathbf{u}).$$
		Hence $c_l\le c'.$
		The proof is complete.
	\end{proof}

	Next, we estimate the energy of semi-trivial solutions of \eqref{sy1}.
	\begin{lemma}\label{max}
		For any semi-trivial solution $\mathbf{u}$ of the system \eqref{sy1}, there holds $$I(\mathbf{w})<I(\mathbf{u})  \ \text{ for every } \ \mathbf{w}\in \hat{\mathbf{H}}(\mathbf{u}), \mathbf{w}\ne \mathbf{u}.$$
	\end{lemma}
	
	\begin{proof} Let $\mathbf{u}$ be a semi-trivial solution of \eqref{sy1}.  For $\mathbf{w}\in \hat{\mathbf{H}}(\mathbf{u})$, we can write  $\mathbf{w}= t \mathbf{u} + \mathbf{v} $ where  $\mathbf{v}\in \tilde{\mathbf{H}}$ and $t\ge0$.
		It is sufficient to consider the case that $\mathbf{u}=(u_1,0)$.
		This means that  $u_2=0$ and $\mathbf{w}=(w_1, w_2)=(tu_1+v_1, v_2)$.  Observe that
		\begin{align*}
			I(\mathbf{w})-I(\mathbf{u})
			&= I(\mathbf{w})-I(\mathbf{u})-I'(\mathbf{u})\left(\frac{t^2-1}{2}\mathbf{u}+t\mathbf{v}\right)\\
			&= -\frac{1}{2}J(\mathbf{v},\mathbf{v})+\int_{\Omega}g(t,\mathbf{u},\mathbf{v}) \dx,
		\end{align*}
		where $$	g(t,\mathbf{u},\mathbf{v})=f(\mathbf{u})\cdot\left(\frac{t^2-1}{2}\mathbf{u}+t\mathbf{v}\right)+F(\mathbf{u})-F(t\mathbf{u}+\mathbf{v}).$$
		To determine the sign of $I(\mathbf{w})-I(\mathbf{u})$, we need to estimate the term $\int_{\Omega}g(t,\mathbf{u},\mathbf{v}) \dx$. In fact,
		\begin{align*}
			g(t,\mathbf{u},\mathbf{v})&=-\sum_{i=1}^2  \frac{\mu_{i}}{4} \left[(u_i^2-w_i^2)^2+2v_i^2u_i^2\right]-\frac{\beta}{2} \left[(u_1^2-w_1^2)(u_2^2-w_2^2)+v_1^2u_2^2+u_1^2v_2^2\right]
			\\&=-\sum_{i=1}^2  \frac{\mu_{i}}{4} \left[(u_i^2-w_i^2)^2+2v_i^2u_i^2\right]-\frac{\beta}{2} \left[-v_2^2(u_1^2-w_1^2)+u_1^2v_2^2\right]
			\\&=-\sum_{i=1}^2  \frac{\mu_{i}}{4} \left[(u_i^2-w_i^2)^2+2v_i^2u_i^2\right]-\frac{\beta}{2}w_1^2v_2^2\le0,
		\end{align*}
		therefore $I(\mathbf{w})\leq I(\mathbf{u})$.
		Clearly, the equality $\int_{\Omega}g(t,\mathbf{u},\mathbf{v}) \dx=0$ holds if and only if $\mathbf{w}=\mathbf{u}$.
		So $I(\mathbf{w})<I(\mathbf{u})$ when $\mathbf{w}\ne \mathbf{u}$.
	\end{proof}
	Let $Q:\mathbf{H}\rightarrow \tilde{\mathbf{H}}$ be the orthogonal projector and $$G:\mathbf{H}\rightarrow \R\times \tilde{\mathbf{H}}, \ \  \quad G(\mathbf{u})=( I'(\mathbf{u})\mathbf{u},Q\nabla I(\mathbf{u})).$$
	
	\begin{lemma} \label{l22}
		For any semi-trivial solution $\mathbf{u}$ of \eqref{sy1}, $G'(\mathbf{u}):\mathbf{H}\rightarrow \R\times \tilde{\mathbf{H}}$ is surjective.
	\end{lemma}
	\begin{proof}
		Let  $\mathbf{u}=(u_1,0)$ be a semi-trivial solution of \eqref{sy1}.  To show that $G'(\mathbf{u}):\mathbf{H}\rightarrow \R\times \tilde{\mathbf{H}}$ is surjective, it  suffices to show that for every $(t,\mathbf{v})\in \R\times\tilde{\mathbf{H}}$, $(t,\mathbf{v})\ne 0$, there exists $\mathbf{w}$ such that
		$$\la G'(\mathbf{u})\mathbf{w}, (t,\mathbf{v})\ra\ne 0.$$
		Take $(t,\mathbf{v})\in \R\times\tilde{\mathbf{H}}$ with $(t,\mathbf{v})\ne 0$ and choose $\mathbf{w}=t\mathbf{u}+\mathbf{v}$.
		Since $\mathbf{u}=(u_1, 0)$ is a solution of \eqref{sy1},  we have that  $I'(\mathbf{u})\mathbf{u}=I'(\mathbf{u})\mathbf{v}=0$. A direct calculation shows that  
		\begin{align*}	
			\la G'(\mathbf{u})\mathbf{w}, (t,\mathbf{v})\ra
			&=\left\la \left(\begin{array}{c}
				\la I''(\mathbf{u})\mathbf{u}, \mathbf{w}\ra +\la I'(\mathbf{u}), \mathbf{w}\ra \\
				 {Q}I''(\mathbf{u})\mathbf{w}
			\end{array}\right), \left(\begin{array}{c}
				t \\ \mathbf{v}
			\end{array}\right)\right\ra\\
			&=t\la I''(\mathbf{u})\mathbf{u}, \mathbf{w}\ra + t\la I'(\mathbf{u}), \mathbf{w}\ra + \la I''(\mathbf{u})\mathbf{w}, \mathbf{v}\ra\\
			&=\langle I''(\mathbf{u})\mathbf{w},\mathbf{w}\rangle\\
			&=\langle I''(\mathbf{u})\mathbf{w},\mathbf{w}\rangle-I'(\mathbf{u})(t^2u_1-2tv_1,0)\\
			&=J(\mathbf{v},\mathbf{v})-\mu_{1}\int_{\Omega}[2u_1^2(tu_1+v_1)^2+v_1^2u_1^2] \dx-\beta\int_{\Omega}u_1^2v_2^2 \dx.
		\end{align*}
		If $J(\mathbf{v},\mathbf{v})<0$, then
		$$\la G'(\mathbf{u})\mathbf{w}, (t,\mathbf{v})\ra\le J(\mathbf{v},\mathbf{v})< 0.$$
		If $J(\mathbf{v},\mathbf{v})=0$ and $\mathbf{v}=0$, then it follows from  $ (t,\mathbf{v})\ne 0$ that
		$$\la G'(\mathbf{u})\mathbf{w}, (t,\mathbf{v})\ra= -2\mu_{1}\int_\Omega t^2u_1^4 \dx<0.$$
		Let $J(\mathbf{v},\mathbf{v})=0$ and $\mathbf{v}\ne0$. If $v_1\ne 0$, then $v_1$ is an eigenfunction corresponding to $\tau_1$.
		By the unique continuation property of $v_1$, we have
		$$\la G'(\mathbf{u})\mathbf{w}, (t,\mathbf{v})\ra\le -\mu_1\int_\Omega v_1^2u_1^2 \dx< 0.$$
		For the case $v_2\ne0$,  $v_2$ is an eigenfunction corresponding to $\tau_2$,   we have $$\la G'(\mathbf{u})\mathbf{w}, (t,\mathbf{v})\ra\le -\beta\int_\Omega u_1^2v_2^2 \dx< 0.$$
		So we have $\la G'(\mathbf{u})\mathbf{w}, (t,\mathbf{v})\ra<0$ in all cases.	
	\end{proof}

	For a semi-trivial solution $\mathbf{u}\in \N$, $\ker(G'(\mathbf{u}))$ is a closed linear subspace of $\mathbf{H}$, so we have the orthogonal decomposition
	$$\mathbf{H}=\ker(G'(\mathbf{u}))\oplus\ker(G'(\mathbf{u}))^\perp .$$
	Since $G'(\mathbf{u})$ is surjective,  the mapping
	$$ G'(\mathbf{u})|_{\ker(G'(\mathbf{u}))^\perp}:\ker(G'(\mathbf{u}))^\perp\rightarrow\R\times \tilde{\mathbf{H}}$$ is bijective, which ensures that the Implicit Function Theorem is applicable. Thus there is a neighborhood $U(\mathbf{u})$ which has a local parametrization in $\N$,  and the tangent space to $\N$ at $\mathbf{u}$ is $$T_\mathbf{u}\mathcal{N}:=(\ker(G'(\mathbf{u}))^\perp)^\perp=\ker(G'(\mathbf{u})).$$  By the proof of Lemma \ref{l22},
	we  see that  $T_\mathbf{u}\mathcal{N} = \ker(G'(\mathbf{u}))=\mathbf{\mathbf{H}}(\mathbf{u})^{\perp}$.
	
	\
	
	\begin{lemma}
		Assume that $\beta>\Lambda$. Then $c'<c_{sem}.$
	\end{lemma}
	\begin{proof}
		By the definition of $\hat{\beta}_1$ in \eqref{1.7}, there exists $U_1\in K_1$ and $h_2 \in H_2^+\backslash\{0\}$ such that
		\begin{equation}\label{2.1}
			\hat{\beta}_1 \le\frac{J_2(h_2,h_2)}{\int_\Omega U_1^2 h_2^2 \dx}< \beta.
		\end{equation}
		We set $\mathbf{U}_1=(U_1,0)$,
		hence
		$$ \langle I''(\mathbf{U}_1)(0,h_2),(0,h_2)\rangle=J_2(h_2,h_2)-\beta\int_{\Omega}U_1^2h_2^2 \dx<0.$$
		Note that $(0, h_2)\in\mathbf{H}(\mathbf{U}_1)^{\bot}= T_{\mathbf{U}_1}\N$,
		there is a $C^1$-curve $\gamma_1:(-\epsilon,\epsilon)\rightarrow \N$ such that $$P(\gamma_1(t))=t(0,h_2)\text{ and } \gamma_1(0)=\mathbf{U}_1, $$ where $P:\mathbf{H}\rightarrow \ker G'(\mathbf{U}_1)$ {is} the orthogonal projector. We also have
		$$\gamma_1'(0)=(0,h_2).$$
		Since $I'(\mathbf{U}_1)=0$, by the Taylor expansion, we have
		\begin{align*}
			I(\gamma_1(t))&=I(\gamma_1(0))+\frac{t^2}{2}\langle I''(\gamma_1(0))\gamma_1'(0),\gamma_1'(0)\rangle+o(t^2)
			\\&=I(\gamma_1(0))+\frac{t^2}{2}\left(J_2(h_2,h_2)-\beta\int_{\Omega}U_1^2h_2^2 \dx\right)+o(t^2).
		\end{align*}
		It follows from  \eqref{2.1} that there is $0<\bar\epsilon<\epsilon$ such that
		\begin{equation}\label{eqcp}
			I(\gamma_1(t))<I(\gamma_1(0))=I(\mathbf{U}_1), \ \ \ \forall \  t \in (-\bar\epsilon, \bar\epsilon) {\backslash\{0\}}.
		\end{equation}
		We claim that  {$\gamma_1|_{(-\bar\epsilon, \bar\epsilon)\backslash\{0\}}\cap \N'\ne\emptyset$. If the claim does not hold, then the curve $\gamma_1|_{(-\bar\epsilon, \bar\epsilon)\backslash\{0\}}\subset \N\backslash \N'$.
			 We consider a sequence $\{\gamma_1(\epsilon_n)\} \subset \N\backslash \N'$ with
			 $\{\epsilon_n\} \subset (-\bar\epsilon, \bar\epsilon)\backslash\{0\} $ and $\epsilon_n\rightarrow 0$ as $n\rightarrow \infty$.}
		By Corollary \ref{maximum},  there are sequences $\{t_n \geqslant 0\} \subset \R$ and $\{\mathbf{v}_n\} \subset \tilde{\mathbf{H}}$ such that
		$$\{t_n\gamma_1(\epsilon_n)^+ +\mathbf{v}_n\} \subset \N'.$$
		Moreover, the set $\{\gamma_1(\epsilon_n)^+\}_n\cup\{\mathbf{U}_1^+\}$ is compact, by Lemma \ref{compact}, the sequences $\{t_n\}$ and $\{\mathbf{v_n}\}$ are bounded, up to subsequences, we may assume that
		$$t_{n} \rightarrow t \ge 0, \ \ \mathbf{v}_n =(v_{n,1}, v_{n,2}) \rightarrow \mathbf{v}=(v_{1}, v_{2}) \  \text{in} \ H_0^1(\Omega)\times H_0^1(\Omega) \ \text{ as } \ n\rightarrow \infty.$$
		Since
		$I(\gamma_1(\epsilon_n))<I(t_n\gamma_1(\epsilon_n)^++\mathbf{v}_n),$
		passing to the limit, we obtain
		$$   I(\mathbf{U}_1)\le I(t\mathbf{U}_1^++\mathbf{v}).$$
		By Lemma \ref{max}, $\mathbf{U}_1$ is the unique global maximum point of $I|_{\hat{\mathbf{H}}(\mathbf{U}_1)}$, we get that $$t\mathbf{U}_1^++\mathbf{v}=\mathbf{U}_1,$$
		which implies $t_n\rightarrow t=1$.
		Since $\ker G'(\mathbf{U}_1)\subset \mathbf{H}^+$, by Implicit Function Theorem, we obtain that
		$$ t_n\gamma_1(\epsilon_n)^++\mathbf{v}_n=\gamma_1(t_n\epsilon_n), \ \ \hbox{for  $n$ large enough}. $$
		This contradicts to  { $\gamma_1|_{(-\bar\epsilon, \bar\epsilon)\backslash\{0\}}\subset \N\backslash \N'$. Thus $\gamma_1|_{(-\bar\epsilon, \bar\epsilon)\backslash\{0\}}\cap \N'\ne\emptyset$.} From the inequality (\ref{eqcp}), we know
		\begin{equation}\label{eq:lessthansemi-1}
			c'=\inf\limits_{\N'}I<I(\mathbf{U}_1).
		\end{equation}
		Similarly, when $\beta > \hat{\beta}_2$, we have that
		\begin{equation}\label{eq:lessthansemi-2}
			c'<I(\mathbf{U}_2).
		\end{equation}
		The desired conclusion follows directly from \eqref{eq:lessthansemi-1},   \eqref{eq:lessthansemi-2} and \eqref{eq:level_csem}. The proof is complete.
	\end{proof}

	\begin{lemma}\label{PS} The functional 	$I$ satisfies the Palais--Smale condition.
	\end{lemma}
	This lemma can be proved in a similar way as \cite[Lemma 3.2]{clapp_solutions_2020}, and we omit the details.
	At last, we are ready to present the proof of Theorem \ref{gs}.
	
	\begin{proof}[Proof of Theorem \ref{gs}]
		We first claim that the nontrivial critical point set $\mathcal{K}$ of $I$ is bounded away from $0$.
		For  $\mathbf{u}=(u_1, u_2) \in  \mathcal{K}$, there holds that
		\begin{equation*}
			I'(\mathbf{u})(u_{1}^+,0)=0 \quad \text{i.e.,}\quad J_1(u_{1}^+,u_{1}^+)=\int_{\Omega}(\mu_1 u_{1}^3u_{1}^+ +\beta u_{1}u_{1}^+u_{2}^2) \dx.
		\end{equation*}
		By the {H\"older's inequality}, we have
		\begin{equation*}
			J_1(u_{1}^+,u_{1}^+)\le\mu_1\|u_{1}\|_{L^4}^3\|u_{1}^+\|_{L^4}+\beta\|u_{1}\|_{L^4}\|u_{1}^+\|_{L^4}\|u_{2}\|_{L^4}^2.
		\end{equation*} It follows, for some constant $C>0$ independent of $\mathbf{u}\in \mathcal{K}$, that $$\|u_{1}^+\|^2_{H^1_0} \le C\|\mathbf{u}\|^4.$$
		In the same way, we also have that $$\|u^-_{1}\|^2_{H^1_0}, \ \ \|u_{2}^+\|^2_{H^1_0}, \ \ \|u^-_{2}\|^2_{H^1_0} \le C\|\mathbf{u}\|^4.$$
		If $\mathbf{H}^0=\{\mathbf{0}\}$, then we have
		$$\|\mathbf{u}\|^2=\|u_{1}^+\|^2_{H^1_0}+\|u^-_{1}\|^2_{H^1_0}+\|u_{2}^+\|^2_{H^1_0}+\|u^-_{2}\|^2_{H^1_0}\le C\|\mathbf{u}\|^4.$$
		Since $\|\mathbf{u}\|\ne0$, it follows that
		$$\|\mathbf{u}\|\ge \frac{1}{\sqrt{C}}.$$
		
		Let $\mathbf{H}^0\ne \{\mathbf{0}\}$. If the claim does not hold,  then there exists $\{\mathbf{u}_n\}\subset \mathcal{K}$ such that
		$$ \|\mathbf{u}_n\| \rightarrow 0, \ \ n\to\infty.$$   We set $\mathbf{w}_n=\mathbf{u}_n/\|\mathbf{u}_n\|$. Then $\|\mathbf{w}_n\|\equiv 1$ for all $n\in \mathbb{N}$ and
		$$\|\mathbf{w}_n^+\|, \ \|\mathbf{w}_n^-\|\le  C \|\mathbf{u}_n\|^2 \rightarrow 0, \ \ \ n\to\infty.$$
		Therefore
		$$\|\mathbf{w}_n^0\|^2=1-\|\mathbf{w}_n^+\|^2-\|\mathbf{w}_n^-\|^2\rightarrow 1, \ \ \ n\to\infty.$$
		Up to subsequence, we assume that $\mathbf{w}_n^0\rightarrow \mathbf{w}^0 \ne \mathbf{0}$ in $\mathbf{H}$ by the equivalence of weak and strong convergence in finite dimensional space. Dividing both sides of $I'(\mathbf{u}_n)\mathbf{u}_n^0=0$ by $\|\mathbf{u}_n\|^4$ and passing to the limit as $n\rightarrow \infty$, we get
		$$0=\int_{\Omega}[\mu_1 (w^0_1)^4+\mu_2 (w^0_2)^4+2\beta (w^0_1w^0_2)^2] \dx>0,$$
		which is a contradiction, thus the claim holds.
		
		We consider a minimizing sequence $\{\mathbf{u}_n\}\subset \mathcal{K}$ which satisfies $I(\mathbf{u}_n)\rightarrow e$. Obviously, $\{\mathbf{u}_n\}$ is a Palais-Smale sequence of $I$ at the level $e$.  By Lemma \ref{PS}, there exists $\mathbf{u}\in \mathbf{H}$ such that, up to a subsequence if necessary,
		$$ \mathbf{u}_n \rightarrow \mathbf{u^*}\ \hbox{in} \ \mathbf{H}, \ \  I(\mathbf{u^*})=e.$$
		Moreover, in view of Lemma \ref{ct}, $c_l$ is a critical value of $I$. It follows that
		$$I(\mathbf{u^*})=e\le c_l\le c'<c_{sem},$$ and  $\mathbf{u}^*\in \mathbf{H}$ is a ground state solution. The proof is complete.
	\end{proof}
	
	We would like to point out that  the parameters $\tau_1$ and $\tau_2$ in Theorem \ref{gs} can be chosen as any given real numbers.
	Although the above proof emphasizes the indefinite case, it is also effective to the definite case by taking $\tilde{\mathbf{H}}=\{\mathbf{0}\}$.
	Therefore, our arguments in the proof is more general in comparing with the known existence results of ground state such as \cite{ambrosetti_stading_2007}.
	
	\section{Further study on the case $\tau_1=\tau_2=\lambda_1$}
	
	In this section we consider the special case $\tau_1=\tau_2=\lambda_1$.  Then the system \eqref{sy1} can be rewritten as follows:
	\begin{equation}\label{sy2}
		\begin{cases}
			-\Delta u_1-\lambda_1 u_1=\mu_1u_1^3+\beta u_1u_2^2 & \text { in } \Omega
			\\-\Delta u_2-\lambda_1 u_2=\mu_2u_2^3+\beta u_1^{2}u_2 & \text { in } \Omega
			\\u_1=u_2=0 & \text { on } \partial\Omega.
		\end{cases}
	\end{equation}
	In this case, $\tilde{\mathbf{H}}=\{(t_1\varphi_1,t_2\varphi_1):\ t_1,t_2\in \R\}$.

	We first investigate the ground state solution for $\beta>0$ small.
	
	\begin{lemma}\label{unique1}
		Assume that $0<\beta<3\sqrt{\mu_1\mu_2}$. Then for every $\mathbf{u}\in \mathbf{H}^+\backslash \{\mathbf{0}\}$, there exists a unique $\mathbf{v}(\mathbf{u})\in \tilde{\mathbf{H}}$ such that	
		\begin{equation}\label{3.2} I(\mathbf{u}+\mathbf{v}(\mathbf{u}))=\max \limits_{\mathbf{z} \in \tilde{\mathbf{H}}}I(\mathbf{u}+\mathbf{z}).\end{equation}
	\end{lemma}
	\begin{proof}
		For any given $\mathbf{u}\in \mathbf{H}^+\backslash \{\mathbf{0}\}$, by the proof of Lemma \ref{compact}, we can find $R>0$ such that
		$$I|_{\mathbf{H}(\mathbf{u})\backslash B_R(\mathbf{0})}< I(\mathbf{u}).$$
		Since the subspace $\tilde{\mathbf{H}}$ has a finite dimension, there exists a $\mathbf{v}(\mathbf{u})\in \tilde{\mathbf{H}}$ such that $$I(\mathbf{u}+\mathbf{v}(\mathbf{u}))= \max \limits_{\mathbf{z} \in \tilde{\mathbf{H}}}I(\mathbf{u}+\mathbf{z}).$$
		
		Next, we show the uniqueness of $\mathbf{v}(\mathbf{u})$. We claim that  any nontrivial critical point $\mathbf{w}$ of $I|_{\mathbf{u}+\tilde{\mathbf{H}}}$ is a local maximum point.  In fact, we have
		$I'(\mathbf{w})(\varphi_1,0)=I'(\mathbf{w})(0,\varphi_1)=0$, and for any $t_1,t_2\in \R$, it holds that
		\begin{align*}
			\langle &I''(\mathbf{w})(t_1\varphi_1,t_2\varphi_1),(t_1\varphi_1,t_2\varphi_1)\rangle
			\\=&-3\int_{\Omega}(\mu_1t_1^2w_1^2+\mu_2t_2^2w_2^2)\varphi_1^2  \dx-\beta\int_{\Omega} (t_1^2w_2^2+t_2^2w_1^2+4t_1t_2w_1w_2)\varphi_1^2  \dx.
		\end{align*}
		When $t_1\ne0$ and $t_2\ne0$, we have
		\begin{align*}
			\langle &I''(\mathbf{w})(t_1\varphi_1,t_2\varphi_1),(t_1\varphi_1,t_2\varphi_1)\rangle
			\\\le& -3\int_{\Omega}(\mu_1t_1^2w_1^2+\mu_2t_2^2w_2^2)\varphi_1^2  \dx
			+2\beta \int_{\Omega}|t_1t_2w_1w_2|\varphi_1^2 \dx
			\\\le& -3\int_{\Omega}\left[\left(\mu_1-\frac{\beta\sqrt{\mu_1}}{3\sqrt{\mu_2}}\right)t_1^2w_1^2+\left(\mu_2-\frac{\beta\sqrt{\mu_2}}{3\sqrt{\mu_1}}\right)t_2^2w_2^2\right]\varphi_1^2  \dx.
		\end{align*}
		When $t_1\ne0$ and $t_2=0$, we have
		\begin{align*}
			\langle I''(\mathbf{w})(t_1\varphi_1,t_2\varphi_1),(t_1\varphi_1,t_2\varphi_1)\rangle
			=-3\mu_1\int_{\Omega}t_1^2w_1^2\varphi_1^2  \dx-\beta\int_{\Omega} t_1^2w_2^2\varphi_1^2  \dx.
		\end{align*}
		When $t_1=0$ and $t_2\ne0$, we have
		\begin{align*}
			\langle I''(\mathbf{w})(t_1\varphi_1,t_2\varphi_1),(t_1\varphi_1,t_2\varphi_1)\rangle
			=-3\mu_2\int_{\Omega}t_2^2w_2^2\varphi_1^2  \dx-\beta\int_{\Omega} t_2^2w_1^2\varphi_1^2  \dx.
		\end{align*}
		Since $0<\beta<3\sqrt{\mu_1\mu_2}$ and $\varphi_1\neq 0$ a.e., in all of the above cases, we have
		$$\langle I''(\mathbf{w})(t_1\varphi_1,t_2\varphi_1),(t_1\varphi_1,t_2\varphi_1)\rangle<0.$$
		Thus, $\mathbf{w}$ is a local maximum point of $I|_{\mathbf{u}+\tilde{\mathbf{H}}}$.
		
		At last, we only need to show that the local maximum point of $I|_{\mathbf{u}+\tilde{\mathbf{H}}}$ is unique. We assume that  $\mathbf{z}\in\mathbf{u}+\tilde{\mathbf{H}}$ is also a local maximum point of $I|_{\mathbf{u}+\tilde{\mathbf{H}}}$.  Without loss of generality, we assume that $I(\mathbf{w})\ge I(\mathbf{z})$. Since $\mathbf{u}+\tilde{\mathbf{H}}$ is a two dimensional manifold, there exists a bounded open neighborhood $U(\mathbf{z})$ of $\mathbf{z}$ in $\mathbf{u}+\tilde{\mathbf{H}}$ such that $\mathbf{w}\in (\mathbf{u}+\tilde{\mathbf{H}})\backslash U(\mathbf{z})$ and
		$$\inf\limits_{\partial U(\mathbf{z})}(-I)>\max\{-I(\mathbf{w}),-I(\mathbf{z})\}.$$
		By Lemma \ref{PS}, $-I|_{\mathbf{u}+\tilde{\mathbf{H}}}$  satisfies the Palais--Smale condition. There is a mountain pass type critical point by \cite[Theorem 4.10]{mawhin_critical:1989}, which contradicts to the claim proved in the second step. Therefore the uniqueness holds and the proof is complete.
	\end{proof}

	\begin{lemma}\label{lem:homogeneity}
		Assume that $0<\beta<3\sqrt{\mu_1\mu_2}$. Then for any $t>0$, $\mathbf{v}(t\mathbf{u})=t\mathbf{v}(\mathbf{u})$.
	\end{lemma}
	\begin{proof} For each $\mathbf{u} \in \mathbf{H}^+ \backslash \{\mathbf{0}\}$, by Lemma \ref{unique1}, there is a unique $\mathbf{v}(\mathbf{u})$ satisfying
		\eqref{3.2}. Note that $J(\mathbf{u},\mathbf{u})=J(\mathbf{u}+\mathbf{v}(\mathbf{u}),\mathbf{u}+\mathbf{v}(\mathbf{u}))$, then $$\int_{\Omega}F(\mathbf{u}+\mathbf{v}(\mathbf{u})) \dx=\min\limits_{\mathbf{v}\in\tilde{\mathbf{H}}}\int_{\Omega}F(\mathbf{u}+\mathbf{v}) \dx.$$ 		
		By homogeneity, for any $t>0$, we have
		\begin{align*}
			\int_{\Omega} F(t\mathbf{u}+\mathbf{v}(t\mathbf{u})) \dx &\le\int_{\Omega} F(t\mathbf{u}+t\mathbf{v}(\mathbf{u})) \dx
			\\&=t^4	\int_{\Omega} F(\mathbf{u}+\mathbf{v}(\mathbf{u})) \dx
			\\&\le t^4\int_{\Omega} F\left(\mathbf{u}+\frac{\mathbf{v}(t\mathbf{u})}{t}\right) \dx
			\\&= \int_{\Omega} F(t\mathbf{u}+\mathbf{v}(t\mathbf{u})) \dx,
		\end{align*}
		which implies every inequality must be equality, so $\mathbf{v}(t\mathbf{u})=t\mathbf{v}(\mathbf{u})$ by the uniqueness.
	\end{proof}
	\begin{lemma} \label{unique2}
		Assume that $0<\beta<3\sqrt{\mu_1\mu_2}$. For every $\mathbf{u}\in \mathbf{H}^+\backslash\{\mathbf{0}\}$, the set $\N'\cap \hat{\mathbf{H}}(\mathbf{u})$
		consists of precisely one point $\hat{\mathbf{u}}$ which is the unique global maximum point of $I|_{\hat{\mathbf{H}}(\mathbf{u})}$. Moreover, $\N=\N'$.
	\end{lemma}
	
	\begin{proof}
		Fix $\mathbf{u}\in \mathbf{H}^+\backslash\{\mathbf{0}\}$. By Corollary \ref{maximum}, we have $\N'\cap \hat{\mathbf{H}}(\mathbf{u})\neq \emptyset$. Since $\N'\subset \N$, it suffices to show that $\N\cap \hat{\mathbf{H}}(\mathbf{u})$ contains at most one element.
		
		By Lemma \ref{unique1} and Lemma \ref{lem:homogeneity}, we see that the elements of $\N\cap \hat{\mathbf{H}}(\mathbf{u})$ must take form
		\begin{equation}\label{eq:lem33-0}
			\{t(\mathbf{u}+\mathbf{v}(\mathbf{u})): \ t>0\},
		\end{equation}
		and satisfy
		\begin{equation}\label{eq:lem33-1}
			I '(t(\mathbf{u}+\mathbf{v}(\mathbf{u}))|_{\tilde{\mathbf{H}}(\mathbf{u})} = 0\quad \text{for some} \ t>0.
		\end{equation}
		We only need to  {show} that there exists a unique $t>0$ such that
		\begin{equation}\label{eq:unique_t}
			I'(t(\mathbf{u}+\mathbf{v}(\mathbf{u})))(\mathbf{u}+\mathbf{v}(\mathbf{u}))=0.
		\end{equation}
		By direct calculation, the function of $t$ on the left-hand side of \eqref{eq:unique_t} has a unique critical point which reads as
		$$ t_{\mathbf{u}+\mathbf{v}(\mathbf{u})}=\left[\frac {J(\mathbf{u}+\mathbf{v}(\mathbf{u}),\mathbf{u}+\mathbf{v}(\mathbf{u}))}{\int_\Omega f(\mathbf{u}+\mathbf{v}(\mathbf{u}))\cdot (\mathbf{u}+\mathbf{v}(\mathbf{u}))\dx}\right]^\frac{1}{2},$$
		that is
		\begin{equation}\label{eq:lem33-2}
			I'(t_{\mathbf{u}+\mathbf{v}(\mathbf{u})}(\mathbf{u}+\mathbf{v}(\mathbf{u}))(\mathbf{u}+\mathbf{v}(\mathbf{u}))=0.
		\end{equation} By \eqref{eq:lem33-0}, \eqref{eq:lem33-1} and \eqref{eq:lem33-2},
		we have that $$ \N\cap \hat{\mathbf{H}}(\mathbf{u})=\left\{t_{\mathbf{u}+\mathbf{v}(\mathbf{u})}(\mathbf{u}+\mathbf{v}(\mathbf{u}))\right\}.$$
		Recalling the fact that for any $\mathbf{u}\in \mathbf{H}^+\backslash\{\mathbf{0}\}$, there holds
		$$\emptyset\neq\N'\cap \hat{\mathbf{H}}(\mathbf{u})\subset \N\cap \hat{\mathbf{H}}(\mathbf{u}),$$ thus $\N'\cap \hat{\mathbf{H}}(\mathbf{u})= \N\cap \hat{\mathbf{H}}(\mathbf{u})$, and consequently $\N=\N'$.
	\end{proof}
	
	To prove the existence and synchronization of the least energy solution, we introduce the notation
	\begin{equation*}
		S'=\inf\limits_{\mathbf{u}\in \N'}\frac{J(\mathbf{u},\mathbf{u})}{\left(\int_{\Omega}f(\mathbf{u})\cdot\mathbf{u} \dx\right)^{1/2}}.
	\end{equation*}
	In the following two lemmas we give the characterizations of the number $S'$ for all  $\beta>0$.
	\begin{lemma}\label{at1}
		If $\beta>0$,  then the infimum $S'$ is achieved.
	\end{lemma}
	\begin{proof} Firstly, we claim that $S'>0$. In fact, for any $\mathbf{u}\in \N'$ and $s>0$, we have $$I(\mathbf{u})=\frac{1}{2}J(\mathbf{u},\mathbf{u})-\int_\Omega F(\mathbf{u}) \dx\ge\frac{1}{2}J(s\mathbf{u}^+,s\mathbf{u}^+)-\int_\Omega F(s\mathbf{u}^+)  \dx.$$
		Since $$\int_\Omega F(s\mathbf{u}^+)=o(\|s\mathbf{u}^+\|^2), \ \ \hbox{as} \  \|s\mathbf{u}^+\|\rightarrow 0,$$
		there exists $r>0$ small enough independent of $\mathbf{u}$ such that
		$$I(\mathbf{u})\ge\inf\limits_{S_r^+}I>0,$$
		where $S_r^+:=\{\mathbf{u}\in \mathbf
		{H}^+: \|\mathbf{u}\|=r\}.$  Since $\N'\subset \N$, for every $\mathbf{u}\in\N'$, we have
		$$J(\mathbf{u},\mathbf{u})=\int_{\Omega}f(\mathbf{u})\cdot\mathbf{u}\dx=4I(\mathbf{u}).$$
		Then
		\begin{equation}\label{eq:lem34-0}
			S'=\inf\limits_{\mathbf{u}\in \N'}\frac{J(\mathbf{u},\mathbf{u})}{(\int_{\Omega}f(\mathbf{u})\cdot\mathbf{u} \dx)^\frac{1}{2}}=\inf\limits_{\mathbf{u}\in {\N'}}\left(\int_{\Omega}f(\mathbf{u})\cdot\mathbf{u} \dx\right)^\frac{1}{2}=\inf\limits_{\mathbf{u}\in {\N'}}(4I(\mathbf{u}))^\frac{1}{2}>0.
		\end{equation}
		Next we show that $S'$ can be achieved. Take a minimizing sequence $\{\mathbf{u}_n\}$ of $S'$, i.e.,
		$$J(\mathbf{u}_n^+,\mathbf{u}_n^+)^\frac{1}{2}=J(\mathbf{u}_n,\mathbf{u}_n)^\frac{1}{2}=\frac{J(\mathbf{u}_n,\mathbf{u}_n)}{(\int_{\Omega}f(\mathbf{u}_n)\cdot\mathbf{u}_n \dx)^\frac{1}{2}}\rightarrow S',$$
		which indicates that the sequence $\{\mathbf{u}_n^+\}$ is bounded in $\mathbf{H}$. 
		{There holds
			\begin{align*}
				\sum_{i=1}^2\mu_i\int_\Omega|\tilde{u}_{n,i}|^4\dx&\leq 2^4\left(\sum_{i=1}^2\mu_i\int_\Omega(|{u}_{n,i}|^4 + |{u}_{n,i}^+|^4)\dx \right)\\
				&\leq 2^4\left(\int_\Omega f(\mathbf{u}_n)\cdot\mathbf{u}_n \dx + C\|\mathbf{u}_n^+\|^4\right).
			\end{align*}
			Combining with the boundedness of $\{\mathbf{u}_n^+\}$ and the fact $\int_{\Omega}f(\mathbf{u}_n)\cdot\mathbf{u}_n \dx\rightarrow S'^2$, we have that $\{\tilde{\mathbf{u}}_n\}$ is bounded in $L^4(\Omega)\times L^4(\Omega)$. Since $\dim \tilde{\mathbf{H}}=2<\infty$, $\{\tilde{\mathbf{u}}_{n}\}$ is also bounded in $\mathbf{H}$.} In conclusion, the sequence $\{\mathbf{u}_n\}$ is bounded in $\mathbf{H}$. Up to subsequences, we have $\mathbf{u}_n\rightharpoonup\mathbf{u}_0=(u_{0,1},u_{0,2})$ in $\mathbf{H}$ and $u_{n,i}\rightarrow u_{0,i}$ in $L^4(\Omega)$. Then $$\int_{\Omega}f(\mathbf{u}_n)\cdot\mathbf{u}_n \dx\rightarrow \int_{\Omega} f(\mathbf{u}_0)\cdot\mathbf{u}_0 \dx\quad \text{as}\ n\to\infty.$$
		Since $\mathbf{u}_n\in \N'$, it holds that for any $\mathbf{v}\in \tilde{\mathbf{H}}$ that
		\begin{equation*}
			\int_{\Omega}f(\mathbf{u}_n)\cdot\mathbf{u}_n \dx\le\int_{\Omega}f(\mathbf{u}_n+\mathbf{v})\cdot(\mathbf{u}_n+\mathbf{v}) \dx,
		\end{equation*}
		passing to the limit we get
		\begin{equation}\label{eq:lem34-1}
			\int_{\Omega}f(\mathbf{u}_0)\cdot\mathbf{u}_0 \dx\le\int_{\Omega}f(\mathbf{u}_0+\mathbf{v})\cdot(\mathbf{u}_0+\mathbf{v}) \dx.
		\end{equation}
		Multiplying $t^4$ on both sides of \eqref{eq:lem34-1} and replacing $t\mathbf{v}$ with $\mathbf{v}$, by using the homogeneity of $f$, we obtain
		\begin{equation}\label{eq:lem34-2}
			\int_{\Omega}f(t\mathbf{u}_0)\cdot (t\mathbf{u}_0) \dx\le\int_{\Omega}f(t\mathbf{u}_0+\mathbf{v})\cdot(t\mathbf{u}_0+\mathbf{v}) \dx,\quad \forall \ t>0, \mathbf{v}\in \tilde{\mathbf{H}}.
		\end{equation}
		
		We claim that $\mathbf{u}_0^+\ne0$.  Otherwise, $\mathbf{u}_0=\tilde{\mathbf{u}}_0$.
		We take $\mathbf{v}=-\tilde{\mathbf{u}}_0$ and $t=1$ in \eqref{eq:lem34-2}, then
		$$\int_{\Omega}f(\tilde{\mathbf{u}}_0)\cdot\tilde{\mathbf{u}}_0 \dx\le\int_{\Omega}f(\tilde{\mathbf{u}}_0+\mathbf{v})\cdot(\tilde{\mathbf{u}}_0+\mathbf{v}) \dx=0.$$
		Thus $\mathbf{u}_0=\tilde{\mathbf{u}}_0=\mathbf{0}$, which contradicts to \eqref{eq:lem34-0}.
		
		To prove that $S'$ is achieved, we only need to show $\mathbf{u}_0\in \N'$.   Since $\mathbf{u}_0^+\ne0$, it follows that the number
		$$t_{\mathbf{u}_0}:=\frac{J(\mathbf{u}_0,\mathbf{u}_0)}{\int_{\Omega}f(\mathbf{u}_0)\cdot\mathbf{u}_0 \dx}>0.$$   By Fatou's Lemma, we have that
		\begin{equation*}\label{eq:lem34-3}
			t_{\mathbf{u}_0}=\frac{J(\mathbf{u}_0,\mathbf{u}_0)}{\int_{\Omega}f(\mathbf{u}_0)\cdot\mathbf{u}_0 \dx}\le \lim\limits_{n\rightarrow\infty}\frac{J(\mathbf{u}_n,\mathbf{u}_n)}{\int_{\Omega}f(\mathbf{u}_n)\cdot\mathbf{u}_n \dx}=1.
		\end{equation*}
		We claim that $t_{\mathbf{u}_0}=1$. If not, then $t_{\mathbf{u}_0}\in(0,1)$. Notice that $t_{\mathbf{u}_0}\mathbf{u}_0$ satisfies \eqref{eq:lem34-2} and
		$$I'(t_{\mathbf{u}_0}\mathbf{u}_0)(t_{\mathbf{u}_0}\mathbf{u}_0)=0,$$
		which implies $t_{\mathbf{u}_0}\mathbf{u}_0\in \N'$. Then
		\begin{align*}
			S'^2\le
			4 I(t_{\mathbf{u}_0}\mathbf{u}_0)& =t^4_{\mathbf{u}_0} \int_{\Omega}f(\mathbf{u}_0)\cdot\mathbf{u}_0 \dx\\
			&< \int_{\Omega}f(\mathbf{u}_0)\cdot\mathbf{u}_0 \dx=\lim\limits_{n\rightarrow \infty}  \int_{\Omega}f(\mathbf{u}_n)\cdot\mathbf{u}_n \dx= S'^2,
		\end{align*}
		which is a contradiction. Thus  {$t_{\mathbf{u}_0}=1$} and $\mathbf{u}_0\in \N'$. It follows from
		$$S'\le  \frac{J(\mathbf{u}_0,\mathbf{u}_0)}{(\int_{\Omega}f(\mathbf{u}_0)\cdot\mathbf{u}_0 \dx)^\frac{1}{2}}\le \lim\limits_{n\rightarrow\infty}\frac{J(\mathbf{u}_n,\mathbf{u}_n)}{(\int_{\Omega}f(\mathbf{u}_n)\cdot\mathbf{u}_n \dx)^\frac{1}{2}}= S'.$$
		Thus $S'$ is achieved at $\mathbf{u}_0$.
	\end{proof}
	Consider the functional $I_*:H_0^1(\Omega)\rightarrow \R$ defined by
	$$I_*(u)=\frac{1}{2}\int_{\Omega}(|\nabla u|^2-\lambda_1 u^2) \dx-\frac{1}{4}\int_{\Omega}u^4 \dx, $$
	and the sets
	\begin{equation*} \begin{array}{ll} &
			\N_*:=\{u\in  H_0^1(\Omega)\backslash \{\R\varphi_1\}:
			I'_*(u)u=0,I'_*(u)\varphi_1=0\}, \\[3mm] &
			\N'_*:=\{u\in  H_0^1(\Omega)\backslash \{\R\varphi_1\} :I_*(u)\ge
			I_*(tu+k\varphi_1)\text{ for all } t\ge0, k\in\R \}. \end{array}
	\end{equation*}
	By \cite[Proposition 2.3]{szulkin_ground_nodate-1}, we get that $\N_*=\N'_*$.
	Define
	\begin{equation*}
		S=\inf\limits_{u\in\N'_*}\frac{\int_{\Omega}(|\nabla u|^2-\lambda_1u^2) \dx}{\| u\|_{L^4}^2}.
	\end{equation*}  
	We give the number $S$ another characterization. For any $u\in H_0^1(\Omega)\backslash \{\R\varphi_1\}$, we consider $$\ell(k):=\frac{1}{4}\int_{\Omega}(u+k\varphi_1)^4 \dx,\ k\in\R.$$ Since $\varphi_1\ne 0 $ a.e., when $k_1\ne k_2$, we have
	\begin{align*}
		&\langle \ell'(k_1)-\ell'(k_2), k_1-k_2\rangle
		\\=&\int_{\Omega}[(u+k_1\varphi_1)^3-(u+k_2\varphi_1)^3](k_1\varphi_1-k_2\varphi_1) \dx
		\\=&\int_{\Omega}(k_1\varphi_1-k_2\varphi_1)^2[(u+k_1\varphi_1)^2+(u+k_1\varphi_1)(u+k_2\varphi_1)+(u+k_2\varphi_1)^2] \dx>0.
	\end{align*}
	By \cite[Theorems 1.5.10]{badiale_semilinear_2010},
	$\ell(k)$ is strictly convex about $k$. Moreover, $\ell(k)$ is coercive, then there exists a unique $k(u)\in \R$ such that
	\begin{equation}\label{eq:def-ku}
		\int_{\Omega}(u+k(u)\varphi_1)^4 \dx=\min\limits_{k\in\R}\int_{\Omega}(u+k\varphi_1)^4 \dx.
	\end{equation}
	{We have $I'_*(u+k(u)\varphi_1)\varphi_1=0,$ and $t_u(u+k(u)\varphi_1)\in \N_*=\N'_*$, where $t_u$ is given by $$t_u=\frac{\int_{\Omega}(|\nabla (u+k(u)\varphi_1)|^2-\lambda_1(u+k(u)\varphi_1)^2) \dx}{\| u+k(u)\varphi_1\|_{L^4}^4}.$$} Therefore
	\begin{align*}
		S&=\inf\limits_{u\in H_0^1(\Omega)\backslash \{\R\varphi_1\}}\frac{t_u^2\int_{\Omega}(|\nabla (u+k(u)\varphi_1)|^2-\lambda_1(u+k(u)\varphi_1)^2) \dx}{t_u^2\| u+k(u)\varphi_1\|_{L^4}^2}
		\\&=\inf\limits_{u\in H_0^1(\Omega)\backslash \{\R\varphi_1\}}\frac{\int_{\Omega}(|\nabla u|^2-\lambda_1u^2) \dx}{\| u+k(u)\varphi_1\|_{L^4}^2}.
	\end{align*}
	Let $U$ be a ground state solution of
	\begin{equation}\label{eq:scalar-lambda1}
		-\Delta u-\lambda_1u=u^3,\ u\in H^1_0(\Omega).
	\end{equation}
	Then we also have
	\begin{equation}\label{least}
		S=\frac{\int_{\Omega}(|\nabla U|^2-\lambda_1U^2) \dx }{\| U\|_{L^4}^2}.
	\end{equation}
	The number $S$ is independent of the choice of $U$. Define an auxiliary function
	\begin{equation}\label{eq:aux-h}
		h(t_1,t_2)=\frac{t_{1}^{2}+t_{2}^{2}}{(\mu_1t_{1}^{4}+\mu_2t_2^{4}+2\beta t_{1}^{2}t_{2}^{2})^\frac{1}{2}}.
	\end{equation}
	\begin{lemma}\label{cp}
		Assume that $\beta>0$. Then we have $S'\ge\inf\limits_{t_1,t_2>0}h(t_1,t_2)S.$ In particular, if $S'=\inf\limits_{t_1,t_2>0}h(t_1,t_2)S$, then there exists a minimizer for $S'$, whose components are proportional.
	\end{lemma}
	\begin{proof}
		By Lemma \ref{at1}, we may assume that $S'$ is achieved at $\mathbf{u}=(u_1,u_2)$. Let
		$$ t_{i}=\left(\frac{\int_{\Omega}|\hat{u}_{i}|^4 \dx}{\int_{\Omega}|U|^4 \dx}\right)^\frac{1}{4}, $$ where $\hat{u}_{i}=u_{i}+k(u_{i})\varphi_1, \ i=1,2$ and $k(\cdot)$ is given by \eqref{eq:def-ku}.
		Then
		\begin{equation}\label{e1}
			\int_{\Omega}|\hat{u}_{i}|^4 \dx=t_{i}^4\int_{\Omega}|U|^4 \dx,
		\end{equation}
		\begin{equation}\label{e2}
			\int_{\Omega}|\hat{u}_{1}\hat{u}_{2}|^2 \dx\le \left(\int_{\Omega}|\hat{u}_{1}|^4 \dx\right)^\frac{1}{2}\left(\int_{\Omega}|\hat{u}_{2}|^4 \dx\right)^\frac{1}{2}=t_{1}^2t_{2}^2\int_{\Omega}|U|^4 \dx,
		\end{equation}
		There exists $t$ such that $t\hat{u}_i\in \mathcal{N}_*'$.
		By the definition of $S$ and \eqref{least}, we get
		\begin{equation}\label{e3}
			\frac{\int_{\Omega}(|\nabla \hat{u}_1|^2-\lambda_1\hat{u}_1^2) \dx}{(\int_{\Omega}|\hat{u}_{1}|^4 \dx)^\frac{1}{2}}\ge\frac{\int_{\Omega}(|\nabla U|^2-\lambda_1U^2) \dx}{(\int_{\Omega}|U|^4 \dx)^\frac{1}{2}},
		\end{equation}
		\begin{equation}\label{e4} \frac{\int_{\Omega}(|\nabla \hat{u}_2|^2-\lambda_1\hat{u}_2^2) \dx}{(\int_{\Omega}|\hat{u}_{2}|^4 \dx)^\frac{1}{2}}\ge\frac{\int_{\Omega}(|\nabla U|^2-\lambda_1U^2) \dx}{(\int_{\Omega}|U|^4 \dx)^\frac{1}{2}}.
		\end{equation}
		By (\ref{e1}), (\ref{e2}), (\ref{e3}) and (\ref{e4}), we obtain
		\begin{align*}
			S'=&\frac{J(\mathbf{u},\mathbf{u})}{(\int_{\Omega}f(\mathbf{u})\cdot\mathbf{u} \dx)^\frac{1}{2}}
			\\=&\frac{\int_{\Omega}(|\nabla \hat{u}_1|^2-\lambda_1\hat{u}_1^2) \dx+\int_{\Omega}(|\nabla \hat{u}_2|^2-\lambda_1\hat{u}_2^2) \dx}{(\int_{\Omega}f(\mathbf{u})\cdot\mathbf{u} \dx)^\frac{1}{2}}
			\\ \ge &\frac{\int_{\Omega}(|\nabla \hat{u}_1|^2-\lambda_1\hat{u}_1^2) \dx+\int_{\Omega}(|\nabla \hat{u}_2|^2-\lambda_1\hat{u}_2^2) \dx}{(\int_{\Omega}f(\hat{u}_{1},\hat{u}_{2})\cdot(\hat{u}_{1},\hat{u}_{2}) \dx)^\frac{1}{2}}
			\\ \ge& \frac{(t_{1}^{2}+t_{2}^{2})\int_{\Omega}(|\nabla U|^2-\lambda_1U^2) \dx}{(\mu_1t_{1}^{4}+\mu_2t_{2}^{4}+2\beta t_{1}^{2}t_{2}^{2})^\frac{1}{2}(\int_{\Omega}|U|^4 \dx)^\frac{1}{2}}
			\\ \ge&\inf\limits_{t_1,t_2>0}h(t_1,t_2)S,
		\end{align*}
		the equality holds only if $\hat{u}_1$ and $\hat{u}_2$ are proportional. Furthermore, if $S'=\inf\limits_{t_1,t_2>0}h(t_1,t_2)S$, then we have that $S'$ is achieved at $t(\hat{u}_1,\hat{u}_2)$, where
		$$t=\frac{\int_{\Omega}(|\nabla \hat{u}_1|^2-\lambda_1\hat{u}_1^2) \dx+\int_{\Omega}(|\nabla \hat{u}_2|^2-\lambda_1\hat{u}_2^2) \dx}{\int_{\Omega}f(\hat{u}_{1},\hat{u}_{2})\cdot(\hat{u}_{1},\hat{u}_{2}) \dx}.$$  In fact, we have
		$$\int_{\Omega}f(\mathbf{u})\cdot\mathbf{u} \dx=\int_{\Omega}f(\hat{u}_{1},\hat{u}_{2})\cdot(\hat{u}_{1},\hat{u}_{2}) \dx=\min\limits_{\mathbf{v}\in\tilde{\mathbf{H}}}\int_{\Omega}f(\mathbf{u}+\mathbf{v})\cdot(\mathbf{u}+\mathbf{v})dx,$$
		which leads to $I'(t\hat{u}_1,t\hat{u}_2)\mathbf{v}=0$ for all $\mathbf{v}\in\tilde{\mathbf{H}}$. By direct calculation, we obtain
		$$I'(t\hat{u}_1,t\hat{u}_2)(t\hat{u}_1,t\hat{u}_2)=0.$$
		So  {$(t\hat{u}_1,t\hat{u}_2)\in \N'$}. Moreover, the equality implies
		$$S'=\frac{\int_{\Omega}(|\nabla \hat{u}_1|^2-\lambda_1\hat{u}_1^2) \dx+\int_{\Omega}(|\nabla \hat{u}_2|^2-\lambda_1\hat{u}_2^2) \dx}{(\int_{\Omega}f(\hat{u}_{1},\hat{u}_{2})\cdot(\hat{u}_{1},\hat{u}_{2}) \dx)^\frac{1}{2}},$$
		then $S'$ is achieved at $(t\hat{u}_1,t\hat{u}_2)$.
	\end{proof}
	\begin{lemma}\label{at2}
		We assume that $0<\beta<3\sqrt{\mu_1\mu_2}$. Then $S'=\inf\limits_{t_1,t_2>0} h(t_1,t_2)S$.
	\end{lemma}
	\begin{proof}
		Let $U$ be a ground state solution of \eqref{eq:scalar-lambda1}, then
		$$\int_{\Omega} U^3\varphi_1 \dx=0.$$
		By direct calculation, we get that $$I'(t_1U,t_2U) (k_1\varphi_1,k_2\varphi_1)=0, \ \text{for any }k_1,k_2\in\R, \  t_1,t_2>0.$$
		Lemma \ref{unique1} implies
		$\mathbf{v}(t_1U,t_2U)=\mathbf{0}$, then there exists $t>0$ such that  $t(t_1U,t_2U)\in \N'$. By definition of $S'$,
		\begin{align*}
			S'&\le \frac{t^2(t_{1}^{2}+t_{2}^{2})\int_{\Omega}(|\nabla U|^2-\lambda_1U^2) \dx}{t^2(\mu_1t_{1}^{4}+\mu_2t_2^{4}+2\beta t_{1}^{2}t_{2}^{2})^\frac{1}{2}(\int_{\Omega}|U|^4 \dx)^\frac{1}{2}}\\
			&=\frac{(t_{1}^{2}+t_{2}^{2})\int_{\Omega}(|\nabla U|^2-\lambda_1U^2) \dx}{(\mu_1t_{1}^{4}+\mu_2t_2^{4}+2\beta t_{1}^{2}t_{2}^{2})^\frac{1}{2}(\int_{\Omega}|U|^4 \dx)^\frac{1}{2}}\\ & = h(t_1, t_2) S, \ \ \ \forall \ t_1,t_2>0.
		\end{align*}
		Hence 	
		$$S'\le\inf\limits_{t_1,t_2>0}h(t_1,t_2)S.$$
		By Lemma \ref{cp}, we see that $S'\ge\inf\limits_{t_1,t_2>0}h(t_1,t_2)S$. Then
		$	S'=\inf\limits_{t_1,t_2>0}h(t_1,t_2)S. 		$
	\end{proof}
	\begin{proof}[Proof of Theorem \ref{inf}]
		By Lemmas \ref{cp} and \ref{at2}, there is a minimizer for $S'$, whose components are proportional.
		Note that $$c=c'=\inf_{\N'}I=\inf_{\mathbf{u}\in\N'}\frac{J(\mathbf{u},\mathbf{u})}{4}=\frac{S'^2}{4}.$$
		Then $c$ and $c'$ are also achieved at the same point.
		Note that the inequality $e\le c_l\le c'$ also holds under the assumptions here. Combining with (\ref{cineq}), we conclude that $c=e=c_l= c'$.
	\end{proof}

	{\begin{remark}
			Denote
			$$s_1^2=\frac{t_1^2}{t_1^2+t_2^2},\quad s_2^2=\frac{t_2^2}{t_1^2+t_2^2},$$
			then by recalling the auxiliary function $h$ defined in \eqref{eq:aux-h}, we obtain
			$$\inf\limits_{t_1,t_2>0}h(t_1,t_2)=\min\limits_{|(s_1,s_2)|=1}\frac{1}{(\mu_1s_1^4+\mu_2s_2^4+2\beta s_1^2s_2^2)^\frac{1}{2}}.$$
			Thus to find $\inf\limits_{t_1,t_2>0}h(t_1,t_2)$ is equivalent to determine $h^*:=\max\limits_{|(s_1,s_2)|=1}\bar{h}(s_1,s_2)$, where $$\bar{h}(s_1,s_2)=\mu_1s_1^4+\mu_2s_2^4+2\beta s_1^2s_2^2.$$
			{By \cite[Theorem 2.1]{correia_semitrivial_2016}, if $\beta\le\max\{\mu_1,\mu_2\}$, then $h^*$ can be achieved by the vector $(s_1, s_2)\in\R^2$ with one trivial component.}
			\begin{enumerate}
				\item[(i)] If $0<\beta\le\max\{\mu_1,\mu_2\}$ and $0<\beta<3\sqrt{\mu_1\mu_2}$, then we have
				$S'=\inf\limits_{t_1,t_2>0}h(t_1,t_2)S$. So there exists a semi-trivial solution to the system (\ref{sy2}) which is a ground state solution.
				\item[(ii)] When $\beta>\max\{\mu_1,\mu_2\}$, it is well known that there exists a synchronized solution, whose energy is strictly less than that of all semi-trivial solutions.
			\end{enumerate}
	\end{remark}}

	In the following, we prove Theorem \ref{cpr} which concerns the ground state solution of \eqref{sy2} for $\beta>0$ large.
	\begin{proof}[Proof of Theorem \ref{cpr}] We claim that $$S'> \inf\limits_{t_1,t_2>0}h(t_1,t_2)S.$$  If the claim does not hold, then Lemma \ref{cp} implies that $S'= \inf\limits_{t_1,t_2>0}h(t_1,t_2)S$, and $S'$ is achieved at $(\alpha_1 u,\alpha_2 u)$ with $u$ being a solution of $-\Delta u-\lambda_1u=u^3$ in $H^1_0(\Omega)$ and
		$$\alpha_1=\sqrt{\frac{\mu_2-\beta}{\mu_1 \mu_2-\beta^2}},\quad \alpha_2=\sqrt{\frac{\mu_1-\beta}{\mu_1 \mu_2-\beta^2}}.$$ According to Remark \ref{rem:1-4}, for $\beta>0$ large enough, $(\alpha_1u,\alpha_2u)\notin\N'$.
		This is a contradiction. Thus the claim holds.
		At last, we have
		$$c\le e\le \frac14 \left(\inf\limits_{t_1,t_2>0}h(t_1,t_2)S\right)^2<\frac14 S'^2=c'.$$
		The proof is complete.
	\end{proof}
	
%
%
	
	\paragraph{Acknowledgments} Supported by NSFC (12271373, 12171326).

\end{document}